\documentclass[12pt, reqno]{amsart}
\usepackage{amssymb,latexsym}
\newcommand{\RNum}[1]{\lowercase\expandafter{\romannumeral #1\relax}}
\usepackage{xcolor}
\usepackage{makecell}
\usepackage{amsthm}
\usepackage{amsmath}
\usepackage{hyperref}
\usepackage{esvect}
\usepackage{mathrsfs}
\hypersetup{colorlinks=true,linkcolor=blue,urlcolor=blue,citecolor=red}
\newtheorem{proof 1}{proof 1}
\newtheorem{theorem}{Theorem}[section]

\newcommand{\beq}{\begin{equation}}
\newcommand{\eeq}{\end{equation}}

\newtheorem{definition}{Definition}[section]
\newtheorem*{equiv*}{Equivalence Relation}

\newtheorem{corollary}{Corollary}[section]

\def\sts#1#2{\begin{Bmatrix}
#1 \\ #2 \end{Bmatrix}}
\title[]{Partial degenerate Stirling numbers}

	\author{Be\'ata B\'enyi}
	\address{\noindent Faculty of Water Sciences, University of Public Service, Baja, HUNGARY}
	\email{benyi.beata@uni-nke.hu}
	
	\author{Sithembele Nkonkobe}
	\address{\noindent School of Mathematics, University of Witwatersrand, 2050 Wits, Johannesburg, South Africa}
	\email{snkonkobe@gmail.com}

\date{\today}

\begin{document}\begin{abstract}
		In this paper, we study some combinations of the degenerate and incomplete Stirling numbers of the second kind. We use a combinatorial approach and provide some asymptotic results. 	\end{abstract}

        \maketitle

\section{Introduction}
 Stirling numbers of the second kind, $\sts{n}{k}$, occur in several areas of mathematics, such as combinatorics, analysis, algebra, probability theory. The basic combinatorial meaning is the number of ways the set $[n] = \{1,2,\ldots, n\}$  can be partitioned into $k$ non-empty subsets. Probably this is the reason for their frequent occurrences throughout the different mathematical branches. The number array can also be defined by the recursion that states for $0<k<n$
 \begin{align}\label{classic_recursion}
\sts{n+1}{k} = k\sts{n}{k} + \sts{n-1}{k}, 
 \end{align}
 with initial conditions $\sts{0}{0} = 1$ and $\sts{n}{0}=\sts{0}{n}=0$ for $n>0$.
The analytical importance is revealed by the generating function
\begin{align}\label{classical_genfun}
		\sum\limits_{n=0}^\infty \sts{n}{k}\frac{x^n}{n!}=\frac{\begin{pmatrix}e^x-1\end{pmatrix}^k}{k!}.
\end{align}
Numerous articles have addressed the introduction and study of generalizations of Stirling numbers for decades \cite{Barati19, broder1984r,hsu1998unified,mansour2011new,munagi2016combinatorial,simsek2007q}. Moreover, efforts to unify different approaches were made again and again in order to clarify the connections and the differences. In this way, the literature of the Stirling numbers is growing like a great tree. In the case of special functions, one direction is the study of the incomplete versions of the functions, the other direction is the study of the degenerate versions of the functions. We focus on the connection of these two concepts for Stirling numbers of the second kind, emphasizing the combinatorial nature of our mathematical objects.

In the next section, we recall these two concepts, then we define the new type of Stirling numbers comprising both of them using a combinatorial approach, and in the remaining sections we investigate the properties of these numbers, in how far the properties of the generalizations may be kept further. 

Why is it worth studying such new generalizations and  combinations of ideas? This is what mathematics always does because the more we play around with conditions, the more we understand not only the underlying object but also the techniques that are useful in their study, and eventually it leads to the development of new techniques. Why this particular topic? As we outlined before, the importance of partitions cannot be questioned in mathematics. 

\section{Preliminaries}
\subsection{Incomplete Stirling numbers}
Restrictions on the size of the subsets in a partition arise in many applications, hence there was an obvious need to consider this in general using combinatorial tools.

The \emph{restricted} Stirling numbers $\sts{n}{k}_{\leq \ell}$ are defined as the number of partitions of $[n]$ into $k$ non-empty subsets such that each subset contains \emph{at most} $\ell$ elements \cite{komatsu2015incomplete,mezo2014periodicity,choi2006reciprocity}. The recursive relation \eqref{classic_recursion} modifies 
\begin{align}\label{restricted_recursion}
	\sts{n+1}{k}_{\leq \ell}=\sum\limits_{i=0}^{\ell-1}\binom{n}{i}\sts{n-i}{k-1}_{\leq \ell}.
\end{align}

Similarly, the number of partitions of $[n]$ into $k$ non-empty subsets such that each subset contains \emph{at least} $\ell$ elements is called the \emph{associated} Stirling number of the second kind and is denoted by $\sts{n}{k}_{\geq \ell}$ \cite{howard1980associated,kim2022degenerate, komatsu2015incomplete,mezo2014periodicity}. The basic recursion and the generating function of associated Stirling numbers of the second kind are as follows. Let $n\geq k\ell$, then

\begin{align*}
	\sts{n+1}{k}_{\geq \ell}=\sum_{i=\ell-1}^{n}\binom{n}{i}\sts{n-i}{k-1}_{\geq \ell},
\end{align*}

Using the following notation of incomplete exponential functions
\begin{align*}
e_{\leq\ell}(x)= \sum_{i =0}^{\ell} \frac{x^i}{i!} \qquad \mbox{and} \qquad e_{<\ell}(x) = \sum_{i=0}^{\ell-1}\frac{x^i}{i!},
\end{align*}
the generating functions for incomplete Stirling numbers of the second kind are \cite{komatsu2015incomplete}
\begin{align*}
	\sum_{n=0}^\infty  \sts{n}{k}_{\leq \ell}\frac{x^n}{n!}=\frac{(e_{\leq \ell}(x)-1)^k}{k!}
    \end{align*}
    \begin{align}\label{equation:26}
	\sum_{n=k\ell}^\infty \sts{n}{k}_{\geq \ell}\frac{x^n}{n!}=\frac{(e^x-e_{<\ell}(x))^k}{k!}.
\end{align}
\subsection{Degenerate Stirling numbers}
We recall now the notion on degenerate special numbers and functions.  Let $\lambda$ be any nonzero real number. 
The \emph{degenerate (or generalized)} falling factorial is defined as  
\begin{align*}
(t)_{n,\lambda} = t(t-\lambda)\cdots(t-(n-1)\lambda), \quad (t)_{0,\lambda} = 1.
\end{align*}
For $\lambda = 1$, we get the falling factorial $(t)_n = (t)_{n,1}= t(t-1)\cdots(t-n+1)$, while $\lim_{\lambda\to 0} (t)_{n,\lambda} = t^n$.  (Other special values for $\lambda$ lead to other often used combinatorial numbers, as for example $\lambda =-1$ gives the increasing factorial.)
The \emph{degenerate} Stirling numbers of the second kind can be defined by the relation
\begin{align*}
(t)_{n,\lambda} = \sum_{k=0}^{n} \sts{n}{k}_{\lambda} (t)_k.
\end{align*}
In order to have the generating function, let us recall the definition of the \emph{degenerate exponential function} \cite{KimKim2020}

\begin{align*}
e_{\lambda}^x(t) = (1+\lambda t)^{\frac{x}{t}} = \sum_{n=0}^{\infty} (x)_{n,\lambda} \frac{t^n}{n!}; \quad e_{\lambda}(t) = e^{1}_{\lambda}(t).
\end{align*}

Note that $\lim_{\lambda\to 0}e^{x}_{\lambda}(t) = e^{xt}$. 
Then we have 
\begin{align*}
\sum_{n=k}^{\infty} \sts{n}{k}_{\lambda}\frac{t^n}{n!} = \frac{(e_{\lambda}(t)-1)^k}{k!}.
\end{align*}
\subsection{Generalized Stirling numbers}
Hsu and Shiue \cite{hsu1998unified} introduced the so called Stirling pairs using a generalization of Stirling numbers as connecting coefficients between polynomials.  Let $\alpha,\beta,\gamma$ be complex numbers and $(\alpha,\beta,\gamma)\not = (0,0,0)$. Let $\sts{n}{k}_{\alpha, \beta, \gamma} $ denote the \emph{generalized Stirling numbers} defined by the relation
\begin{align*}
(t)_{n,\alpha} = \sum_{k=0}^{n}\sts{n}{k}_{\alpha, \beta, \gamma} (t-\gamma)_{k,\beta}.
\end{align*}
We adopt the notation introduced by Maltenfort \cite{Maltenfort2020} instead of the original $S(n,k,\alpha,\beta,\gamma)$ notation.
Then, the generating function for $\sts{n}{k}_{\alpha, \beta, \gamma}$ is  \cite{hsu1998unified}
 \begin{align*}
 \sum_{n=0}^{\infty} \sts{n}{k}_{\alpha, \beta, \gamma}\frac{t^n}{n!} = \frac{(e^{\beta}_{\alpha}(t)-1)^k}{\beta^kk!}e^{\gamma}_{\alpha}(t). 
 \end{align*}

Generalized Stirling numbers involve an amount of special generalizations that have arisen earlier (or later) as special cases. They were studied from many points of view,  combinatorial and probabilistic models were given, and combinatorial, analytical, and algebraic properties were derived.
The recursion \eqref{classic_recursion} has the form in the case of the generalized Stirling numbers as follows:

\begin{align*}
\sts{n+1}{k}_{\alpha, \beta, \gamma} = \sts{n}{k-1}_{\alpha, \beta, \gamma} + (k\beta -n\alpha + \gamma) \sts{n}{k}_{\alpha, \beta, \gamma}.
\end{align*}
Maltenfort \cite{Maltenfort2020} used this recursion to extend the set of parameter values, i.e., for negative $n$ and $k$ values, and for $\alpha$, $\beta$, $\gamma$ lying in any ring.  
Corcino \cite{corcino2001some} derived an elegant formula    for $n\geq 0$
\begin{align*}
\sts{n}{k}_{\alpha,\beta,\gamma} = \frac{\Delta^k(\beta t+ \gamma)_{\alpha,n}|_{t=0}}{\beta^kk!},
\end{align*}
where $\Delta^k$ denotes the forward difference. Considering the definition of forward difference, $\Delta^ka_n = \sum_{i=0}^{k} (-1)^{i}\binom{k}{i}a_{n+k-i}$, we see that the above formula is the translation of the relation:
\begin{align*}
 \sts{n}{k}_{\alpha,\beta,\gamma} =\frac{1}{\beta^kk!}\sum_{j=0}^{k}(-1)^{k-j}\binom{k}{j}(\beta j+\gamma)_{n,\alpha}.
 \end{align*}
The interested reader can dive into the deep sea of beautiful identities that can be found in the wide literature of generalized Stirling numbers. See, for example \cite{benyi2022unfair,corcino2001combinatorial,corcino2001some,kargin2018higher,Maltenfort2020,schork2024stirling} and references therein.

We recall one of the basic combinatorial models presented in \cite{benyi2022unfair}, which we want to extend in this current work.   
First, in order to talk about the properties of the partition easier, we introduce some simple expressions. We distribute elements into cells that have inner structures. 
Cells are separated into compartments and each compartment can contain only $1$ element. We refer to the number of compartments in a cell as \emph{order} of the cell.
We say that a cell has the \emph{$\alpha$-closing property} if their compartments have cyclic ordered numbering, each element is placed at a time, and for each available $\alpha$ compartments only the first compartment may be occupied by an element. 
\begin{definition}\cite{benyi2022unfair}\label{definition:10} Let $\alpha, \beta, \gamma, n, k$ be non-negative integers such that $\alpha|\gamma$, $\alpha|\beta$, and $k\leq n$.  
    By a $(\alpha, \beta, \gamma)$-partition we mean a partition of $[n]$ into $k+1$ cells with $\alpha$-closing properties such that 
    \begin{itemize}
    \item[1.] there is a cell of order $\gamma$ that may be empty,
    \item[2.] there are $k$ cells of order $\beta$ that are non-empty and each cell has a ,,favorite" compartments, i.e., the first element in the cell goes to this specific compartment. 
    \end{itemize}
    The generalized Stirling numbers, $\sts{n}{k}_{\alpha, \beta, \gamma}$, count the number of $(\alpha, \beta, \gamma)$-partitions of $[n]$ into $k+1$ subsets \cite{benyi2022unfair}.
\end{definition}
In the following proofs, we refer to the cell with $\gamma$ compartments as \emph{special cell} and to the cells with $\beta$ compartments as \emph{ordinary cells}.
We use the notation $\sts{n}{k}_{\alpha, \beta}$ for the special case $\gamma = 0$ in the remainder of the paper. 
 Throughout the literature, we have models that are not substantially different from this model; one can view them as reformulations. For more details, examples, and explanations and related studies, see \cite{Adell23, benyi2022unfair}. 

\subsection{Weighted mixed partitions}
We now introduce a model to interpret the generalized Stirling numbers which was not defined yet in the literature in this form. However, the idea is not novel. We explain what we mean below. 

Let $\mathcal{A}$ be a ring, $\alpha, \beta, \gamma$ elements in $\mathcal{A}$, $n$ and $k$ be integers. Let $(G, P_k)$ be a pair where $G$ is a possible empty set and $P_k=B_1/B_2/\ldots/B_k$ is a set partition with $k$ non-empty blocks. We associate the weights with the sets $G$ and $B_i$s as follows. The weight of $G$ is $(\gamma)_{|G|,\alpha}$, where $|G|$ denotes the size of $G$, i.e., the number of elements in $G$. Set $w(G)=1$, if $G$ is empty. The weight of a block $B_i$ is $w(B_i) = (\beta-\alpha)_{|B_{i}|-1,\alpha} = \frac{(\beta)_{|B_i|,\alpha}}{\beta}$. Note that $w(B) =1$ for a singleton block $B$. The weight of partition $P_k$ is the product of the weights of its blocks, i.e., $w(P_k)=\sum_{i=1}^{k} w(B_i)$. Set $w(P_0)=1$ for the empty partition. Let the weight of a pair $(G,P_k)$ defined as $w((G,P_k)) = w(G)w(P_k)$.
 Let $\Pi^*_{n,k}$  denote all distributions of $[n]$ in a possibly empty set, $G$,  and $k$ non-empty blocks, $P_k$. Then, the following combinatorial definition coincides with the generalized Stirling numbers 

\begin{align*}
\sts{n}{k}_{\alpha,\beta,\gamma} = \sum_{(G,P_k)\in \Pi^*_{n,k}} w((G,P_k)).
\end{align*}
As a demonstration of the usefulness of this definition, we recover some  statements about generalized Stirling numbers. 

\begin{itemize}
\item $\sts{n}{k}_{\alpha,\beta,\gamma}=0$, if $n<k$ and $\sts{n}{k}_{\alpha,\beta,\gamma}=1$, if $n=k$, since in this case we have a partition such that each block in $P_n$ is a singleton and $G$ is empty. Both the empty set $G$ and any singleton block have weight $1$. 
\item $\sts{n}{0}_{\alpha,\beta,\gamma} =(\gamma)_{n,\alpha}$, since if we have $k=0$, $w(P_0)=1$, and all elements are in the set $G$.
\item $\sts{n}{1}_{\alpha,\beta} = (\beta-\alpha)_{n-1,\alpha}$, since $G$ is empty and there is only one non-empty block of size $n$. 
\item $\sts{n+1}{k}_{\alpha,\beta}=\sts{n}{k-1}_{\alpha,\beta,\beta-\alpha}$. 
The left hand side of the identity count partitions of $[n+1]$ into $k$ non-empty blocks; pairs $(G, P_k)$ such that $G$ is empty. Consider the block $B$ that contains $1$. $w(B)= (\beta-\alpha)_{|B|-1,\alpha}$. Now, delete $1$ from it, $B' = B\setminus \{1\}$. Further, if we delete the block $B$ from the set partition $P_k$, we obtain a partition $P'_{k-1}$.  Together, $(B', P'_{k-1})$ construct a pair on $n$ elements such that $\gamma = \beta-\alpha$. 
\item $\sts{n}{n-1}_{\alpha,\beta,\gamma}= n\gamma+\binom{n}{2}(\beta-\alpha)$. The right hand side of this identity 
counts the pairs according to whether $G$ is empty or not. If $G$ is not empty, all sets contain only one element. There are $n$ ways to choose the element in $G$, and this set has the weight $\gamma$. On the other hand if $G$ is empty, then one of the blocks in $P_k$ has two elements, the others are singletons. Only the block with $2$ elements has weight other than 1, namely $\beta-\alpha$. There are $\binom{n}{2}$ possibilities to choose these two elements.  
\item $\sts{n}{k}_{c\alpha,c\beta,c\gamma} = c^{n-k}\sts{n}{k}_{\alpha,\beta,\gamma}$, since changing the parameters $\alpha$, $\beta$, $\gamma$ by a multiplication factor $c$ has an effect on the total weight as a multiplication factor $c^{n-k}$.  
\end{itemize}

The idea of mixed weighted partitions roots in previous results. Maltenfort \cite{Maltenfort2020} extends $\alpha, \beta,\gamma$ values to any ring. The authors \cite{Yaqubi2016} introduced \emph{mixed partition numbers} and as a particular case \emph{mixed Stirling numbers of the second kind}, the papers \cite{Barati19, BenyiYaqubi2019} investigate them in more detail. Finally, the notion of partition enumeration  with a given weight associated with the blocks goes back to Carlitz \cite{Carlitz_weighted}, who defined a version of Stirling numbers with $\lambda^j$ as a weight on a $j$ element block. The authors \cite{Belbachir2014, Belbachir2021} make use of weights for their studies and an even more general idea is the \emph{vector weighted Stirling numbers} introduced by the authors of \cite{Esmaeeli2021}. 

The connection with $(\alpha, \beta, \gamma)$-partitions is that the weights correspond to the inner structure of the cells. In this sense, the mixed weighted partitions are abstracter compared to the $(\alpha,\beta,\gamma)$-partitions. However, this approach makes enumerations in certain cases easier because it is closer to the original set partition definition.

\section{Incomplete degenerate Stirling numbers of the second kind}
We consider now two possibilities to combine the two concepts, degenerateness and incompleteness of Stirling numbers of the second kind. 
\subsection{Degenerate restricted Stirling numbers of the second kind}

\begin{definition}\label{definition:3} Let $\sts{n}{k}_{\alpha, \beta, \gamma}^{\leq \ell}$ denote the number of $(\alpha,\beta,\gamma)$-partitions of $[n]$ in $k+1$ cells such that $k$ ordinary cells contain at most $\ell $ elements.  
\end{definition}

Alternatively, let $\Pi^{\leq \ell}_{n,k}$  denote all distributions of $[n]$ into a possibly empty set, $G$,  and $k$ non-empty blocks, $P_k$, such that each block has cardinality at most $\ell$. Then, we have

\begin{align*}
\sts{n}{k}_{\alpha,\beta,\gamma}^{\leq \ell}= \sum_{(G,P_k)\in \Pi^{\leq \ell}_{n,k}} w((G,P_k)).
\end{align*}

First, let us show the fundamental recursive relation,
\begin{theorem}\label{theorem:2_1} We have

\begin{align}\label{basic_recursion}
\sts{n+1}{k}_{\alpha, \beta, \gamma}^{\leq \ell} = \gamma \sts{n}{k}_{\alpha, \beta, \gamma-\alpha}^{\leq \ell} + \sum_{i=k-1}^{n-\ell-1}\binom{n}{i}(\beta-\alpha)_{n-i,\alpha}\sts{i}{k-1}_{\alpha, \beta,\gamma}^{\leq \ell}.
\end{align}
\end{theorem}
\begin{proof}
The left hand side counts $(\alpha,\beta,\gamma)$-partitions of $[n+1]$ into $k+1$ cells. 
The element $n+1$ is either in the special cell or in one of the ordinary $k$ cells. If it is in the special cell, we choose the compartment in $\gamma$ ways, and the remaining elements can be partitioned in $\sts{n}{k}_{\alpha, \beta, \gamma-\alpha}^{\leq \ell}$ ways. On the other hand, if the element $n+1$ is in an ordinary cell, let $i$ be the number of elements that are not in the same cell as $n+1$, choose them in $\binom{n}{i}$ ways. Furthermore, assume $n+1$ is in the favorite compartment, then distribute the $n-i$ elements that join the cell with $n+1$ in $(\beta-\alpha)(\beta-2\alpha)\cdots(\beta-(n-i)\alpha)$ ways. Finally, the $i$ elements construct a $(\alpha, \beta, \gamma)$-partition into $k-1$ non-empty ordinary cells.
\end{proof}

The only substantial difference between identity \eqref{restricted_recursion}  and its generalized version \eqref{basic_recursion} is in the part dealing with the special cell and in an additional factor $(\beta-\alpha)_{n-i,\alpha}$ which arises because of the additional inner structure of a cell.

Let us denote the incomplete degenerate exponential function as follows:
\begin{align*}
e_{\lambda;\leq \ell}^x (t) = \sum_{n=0}^{\ell} (x)_{n,\lambda}\frac{t^n}{n!}.
\end{align*}
The generating function of $\sts{n}{k}_{\alpha, \beta, \gamma}^{\leq \ell}$ is given as follows
\begin{align*}
\sum_{n=0}^{\infty} \sts{n}{k}_{\alpha, \beta, \gamma}^{\leq \ell} \frac{x^n}{n!} = e^{\gamma}_{\alpha}(x) \frac{\left(e^{\beta}_{\alpha;\leq \ell}(x)-1\right)^k}{\beta^kk!}.
\end{align*}
Note that for $\gamma =0$, we have 
\begin{align*}
\sum_{n=0}^{\infty} \sts{n}{k}_{\alpha, \beta}^{\leq \ell} \frac{x^n}{n!} = \frac{\left(e^{\beta}_{\alpha;\leq \ell}(x)-1\right)^k}{\beta^kk!}.
\end{align*}

Let $A(x)$ denote the generating function for short.  
The first derivative of the generating function according to the variable $x$
\begin{align*}
\frac{d}{dx}A(x) = \sum_{n=0}^{\infty} \sts{n+1}{k}_{\alpha, \beta, \gamma}^{\leq \ell} \frac{x^n}{n!}.
\end{align*}
On the other hand, 
\begin{align*}
\frac{d}{dx}A(x) = \gamma e^{\gamma-\alpha}_{\alpha}(x)\frac{(e^{\beta}_{\alpha;\leq \ell}(x)-1)^k}{\beta^kk!} + e^{\gamma}_{\alpha}(x)\frac{(e^{\beta}_{\alpha;\leq \ell}(x)-1)^{k-1}}{\beta^{(k-1)}(k-1)!}\frac{d}{dx}\frac{(e^{\beta}_{\alpha; \leq \ell}(x)-1)}{\beta},
\end{align*}
where 
\begin{align*}\frac{d}{dx}\frac{(e^{\beta}_{\alpha; \leq \ell}(x)-1)}{\beta} = \sum_{n=0}^{\infty}\sts{n+1}{1}_{\alpha,\beta}^{\leq \ell}\frac{x^n}{n!},
\end{align*}
and since 
\begin{align*}
\sts{n}{1}_{\alpha, \beta}^{\leq \ell} =\left\{\begin{matrix} 0, & \mbox{if  } n>\ell\\
(\beta-\alpha)_{n;\alpha} & \mbox{if  } n\leq \ell.
\end{matrix}\right.
\end{align*} we end up with another proof of the recursion in Theorem \ref{theorem:2_1}. 

A three term recurrence is given in the next theorem.
\begin{theorem} We have
\begin{align*}
\sts{n}{k}_{\alpha,\beta, \gamma}^{\leq \ell} &= \gamma\sts{n}{k}_{\alpha,\beta,\gamma-\alpha}^{\leq \ell} + \gamma\sum_{i=k-1}^{\ell}\binom{n}{i}(\beta-\alpha)_{n-i+1,\alpha}\sts{i-1}{k-1}_{\alpha, \beta, \gamma-\alpha}^{\leq \ell}\\&+ \sum_{i=k-1}^{\ell} \binom{n}{i}(\beta-\alpha)_{n-i+1, \alpha}\sum_{j=0}^{i-1}\binom{i-1}{j}(\beta-\alpha)_{i-j,\alpha}\sts{j}{k-2}_{\alpha,\beta,\gamma}^{\leq \ell}.
\end{align*}
\end{theorem}
\begin{proof} Algebraically, we get the theorem by applying the previous identity twice. In the following, we provide a combinatorial argument. 

\textbf{Case 1:} The $(n+1)$th element goes into the special cell. There are $\gamma$ ways to choose a compartment for the $(n+1)$th element which closes $\alpha-1$ further compartments, hence; the remaining $n$ elements form an $(\alpha, \beta, \gamma -\alpha)$-partition which can be done in $\sts{n}{k}_{\alpha, \beta, \gamma-\alpha}^{\leq \ell}$ ways.
	
\textbf{Case 2:} The $(n+1)$th element goes in an ordinary cell together with $n-i$ other elements. Also, from the $i$ remaining elements the maximal element (call it $\epsilon$) is in the cell having $\gamma$ compartments. There are $\binom{n}{i}$ ways to choose $i$ elements that are not in the same cell as $(n+1)$. The $n-i$ other elements, and $(n+1)$ may be distributed in the cell in $\sts{n-i+1}{1}_{\alpha,\beta}^{\leq \ell}$ ways. As from the $i$ elements $\epsilon$ is in the cell having $\gamma$ compartments, there are $\gamma$ ways of placing $\epsilon$ in the cell. The remaining $i-1$ elements may be distributed in the $k$ cells including the one with $\gamma$ compartments in $\sts{i-1}{k-1}_{\alpha,\beta,\gamma}^{\leq \ell}$ ways. 

\textbf{Case 3:}  The element $n+1$ is in one of the ordinary  $k$ cells each together with $n-i$ other elements. Also, from the $i$ remaining elements the maximal element ($\epsilon$) is also in an ordinary cell.  The $n-i$ other elements, together with $n+1$ may be distributed  in $\sts{n-i+1}{1}_{\alpha,\beta}^{\leq \ell}$ ways. From the $i$ elements besides $\epsilon$ we can choose in $\binom{i-1}{j}$ ways, $j$ elements which are not in the same cell as $\epsilon$. The $j$ elements may be distributed in $k-1$ cells including the special cell in $\sts{j}{k-2}_{\alpha,\beta,\gamma}^{\leq \ell}$ ways. Now, $\epsilon$ together with the $i-j-1$ other elements may be distributed into the cell in $\sts{i-j}{1}_{\alpha,\beta}^{\leq \ell}$ ways.  
\end{proof}

\subsection{Generalized associated Stirling numbers}
Similarly, we can define $\sts{n}{k}_{\alpha,\beta,\gamma}^{\geq \ell}$  as the number of $(\alpha,\beta,\gamma)$-partitions such that ordinary cells contain at least $\ell$ elements, or as mixed weighted partitions, i.e., the distribution of $[n]$ into $(G, P_k)$ pairs such that the blocks of $P_k$ contain at least $\ell$ elements. We omit the details of the study of these numbers. 
We turn our attention to partitions into cells with unlimited compartments. We call a cell whose compartments are unlimited a \emph{free cell}. (Actually, free cells correspond to the special case $\alpha =0$, $\beta =1$). Otherwise, as before, a cell has compartments limited to $1$ element. 
We introduce and investigate the following generalization od Stirling numbers in this section. 

\begin{definition}
Let $n,k,\ell,\gamma$ non-negative integers such that $k,\ell\leq n$. Let $\sts{n}{k}_{\gamma}^{>\ell}$ denote the number of partitions of $[n]$ in $k$ non-empty free cells containing at least $\ell+1$ elements and a possible empty free cell with $\gamma$ compartments.  
\end{definition}

We can verify that $\sts{n}{k}^{> \ell}_\gamma$ counts mixed weighted partitions $(G, P_k)$ of $[n]$ with weight $w((G,P_k)) = \gamma^{|G|}$ such that each block in $P_k$ has more than $\ell$ elements.

\begin{theorem}\label{theorem:3}We have
	\begin{align}\label{equation:24}
		\sts{n}{k}_{\geq \ell} = \sum_{i=0}^n(-1)^{i}\gamma^{i}\binom{n}{i}\sts{n-i}{k}^{>\ell-1}_\gamma.
	\end{align}
\end{theorem}
\begin{proof}The weight of all partitions of $[n]$  having at least $i$ elements in $G$ is $\binom{n}{i}\gamma^{i}\sts{n-i}{k}^{>\ell-1}_\gamma$. After applying the inclusion-exclusion principle, we have \ref{equation:24}.  
	
\end{proof}	
By \eqref{equation:24} and \eqref{equation:26} we have
\begin{corollary}\label{corollary:1}

	\begin{align*}\label{equation:4}
		\sum\limits_{n=kl}^\infty \sts{n}{k}^{>\ell}_\gamma\frac{x^n}{n!}=\frac{e^{\gamma x}\left(e^x-e_{\leq\ell}(x)\right)^k}{k!}.
	\end{align*}
\end{corollary}

\begin{theorem}
We have 
\begin{align*}
\sts{n+1}{k}_{\gamma}^{>\ell} = \gamma\sts{n}{k}_{\gamma}^{>\ell} + \sum_{i=0}^{n} \gamma^{i}\binom{n}{i}\sts{n-i}{k}_{\geq \ell+1}.
\end{align*}
\end{theorem}
\begin{proof}
Consider a mixed weighted partition of $n+1$ elements such that each block has more than $\ell$ elements and weight $w((G, P_k)=\gamma^{|G|})$. If $n+1$ is in $G$,  the remaining $n$ elements form a mixed weighted partition, $P'$ with weight $\sts{n}{k}_{\gamma}^{> \ell}$. Putting the element $n+1$ into $G$, the weight $w(P')$ increases by a factor $\gamma$.  On the other hand, if $n+1$ is in a block of $P_k$, the weight will not change. Let $i$ be the size of $G$. Then we have $w(G) = \gamma^i$.  The remaining $n-i$ elements can be partitioned into blocks of size at most $\ell$ in $\sts{n-i}{k}_{\geq \ell+1}$ ways. 
\end{proof}

The numbers $\sts{n}{k}_{\gamma}^{>\ell}$ are closely related to \emph{associated $r$-Stirling numbers}. These numbers count set partitions into non-empty subsets such that certain distinguished elements, say, $\{1,2,\ldots, r\}$ are in different blocks and each block has more than $\ell$ elements. Enumeration results can be found in \cite{BenyiMendez2019}. 

However, these two types of partition do not coincide, since in our definition there is no condition on the size of $G$. A corresponding definition would be the following. Set partitions of $[n+r]$ into $k+r$ blocks such that $\{1,2,\ldots, r\}$ are in distinct blocks, the blocks that do not contain distinguished elements contain at least $\ell$ elements while the blocks containing a distinguished element have at least $1$ element. Such partitions are enumerated by $\sts{n}{k}_r^{>\ell-1}$.

\section{Partial degenerate partitions}\label{section:4}
In the previous sections, we presented some concepts of  generalizations and some examples how these concepts can be combined. In this section, we demonstrate a new idea for combining these notions. We emphasize that our definition is one variation of the many possibilities.

\begin{definition}\label{definition:4}
Let $n,k,\ell, \alpha,\beta,\gamma$ non-negative integers such that $\alpha|\gamma$, $\alpha|\beta$, $k\leq n$, $\ell\leq n$. 
Consider the distribution of $[n]$ into $k+1$ cells that satisfies the following conditions. 
\begin{itemize}
	\item[1.] There is a free cell of order $\gamma$ that may be empty. 
	\item[2.] There are $k$ non-empty cells of order 1 or of order $\beta$. 
	\item[3.] Cells of order $1$ are free and have more than $\ell$ elements
	\item[4.] Cells of order $\beta$ have the $\alpha$-closing property and contain at most $\ell$ elements.
\end{itemize}
We call  these partitions \emph{$\ell$-restricted $(\alpha,\beta,\gamma)$-partial degenerate partitions} and denote by $S^{\gamma,\alpha,\beta}_{n,k,\ell}$ their number.
\end{definition}

In short, we will call these partitions as \emph{partial degenerate partitions}. 

Note that in the language of weighted mixed partition $S^{\gamma,\alpha,\beta}_{n,k,\ell}$ count partitions of $[n]$ into pairs $(G, P_k)$, where $w(G) = \gamma^{|G|}$, and $w(B_i) = 1$, if $|B_i|\leq \ell$ and $w(B_i) = (\beta-\alpha)_{|B_i|-1,\alpha}$ if $|B_i|>\ell$.

 \begin{theorem}\label{theorem:13}For $ n,k, \ell,\alpha, \beta, \gamma$ non-negative integers, we have
	\begin{equation}\label{equation:5}
S^{\gamma,\alpha,\beta}_{n,k,\ell}=\sum\limits_{i=0}^n\sum\limits_{j=0}^k\binom{n}{i}\sts{i}{j}_\gamma^{> \ell}\sts{n-i}{k-j}_{\alpha,\beta}^{\leq \ell}. 
	\end{equation}
\end{theorem}
\begin{proof}
Let $i$ the number of elements in free cells. Let us choose these elements in $\binom{n}{i}$ ways. Let $j$ be the number of free  cells (not of order $\gamma$). Then we can form from these $i$ elements in ${i\brace j}_\gamma^{> \ell}$ ways free cells in our partition. The remaining $n-i$ element however are distributed into the cells of order $\beta$ with $\alpha$-closing property and restriction on the size in $\sts{n-i}{k-j}_{\alpha,\beta}^{\leq \ell}$ ways.
\end{proof}

We can recover a connection to a recently studied type of partitions. Namely, in \cite{mansour2024set} the following generalization of the Stirling numbers 

\begin{align*}
	\sum_{n=0}^\infty S^{(r,s)}_{n,k}\frac{x^n}{n!}=\frac{e^{rx}}{k!}\left(
		e^x+(s-1)x-1
	\right)^k.
\end{align*}
was introduced as the number of partitions of $[n+r]$ into $k+r$ subsets such that non-special singleton blocks are colored with one of $s$ colors. 
In our combinatorial model for $S^{\gamma,\alpha,\beta}_{n,k,\ell}$ elements occupying compartments can alternately be interpreted as elements colored by a certain number of colors. In Definition~\ref{definition:4} for $(\alpha,\beta)=(0,1)$ when we disregard the condition ``the first element in each cell goes into a favorite compartment" we can easily establish a one to one correspondence between the partitions with colored singleton whose number is $S^{(r,s)}_{n,k}$ discussed in \cite{mansour2024set} and our model of partial degenerate partitions. We omit the details.  

\begin{theorem}For Let $n,k,\ell, \alpha,\beta,\gamma$ non-negative integers such that $\alpha|\gamma$, $\alpha|\beta$, $k\leq n$, $\ell\leq n$, we have
	\begin{align*}
		S^{\gamma,\alpha,\beta}_{n+1,k,\ell}= \sum_{i=0}^n\sum_{j=0}^k\binom{n}{i}\left[\sts{i+1}{j}_\gamma^{> \ell}\sts{n-i}{k-j}_{\alpha,\beta}^{\leq \ell}+ \sts{i}{j}_{\gamma}^{>\ell}\sts{n-i+1}{k-j}_{\alpha,\beta}^{\leq \ell}\right].
	\end{align*}
\end{theorem}
\begin{proof}The proof is based on the position of the $(n+1)$st element. Let $i$ be the number of elements that are distributed in free cells besides $n+1$ and let us choose these elements in $\binom{n}{i}$ ways.  If $n+1$ is also in a free cell, let $j$ be the number of free cells, then the free cells can be formed in $\sts{i+1}{j}_{\gamma}^{> \ell}$ ways, while the remaining $n-i$ elements can construct the remaining $k-j$ cells of order $\beta$ in $\sts{n-i}{k-j}_{\alpha,\beta}^{\leq \ell}$ ways. If $n+1$ is in a cell of order $\beta$, the free cells can be formed in $\sts{i}{j}_{\gamma}^{>\ell}$ and the cells of order $\beta$ in $\sts{n-i+1}{k-j}_{\alpha,\beta}^{\leq \ell}$. 
\end{proof}

Next, we provide the generating function.

\begin{theorem}\label{theorem:20}For $n,k,\ell, \alpha,\beta,\gamma$ non-negative integers such that $\alpha|\gamma$, $\alpha|\beta$, $k\leq n$, $\ell\leq n$,  we have
	\begin{align}\label{equation:1}
	\sum_{n=0}^{\infty}S^{\gamma,\alpha,\beta}_{n,k,\ell} \frac{x^n}{n!}=\frac{e^{\gamma x}}{k!}\left(e^x+\sum_{i=1}^{\ell}\frac{(\beta|\alpha)_i}{\beta}\frac{x^i}{i!}-\sum_{i=0}^{\ell}\frac{x^i}{i!}\right)^k.
\end{align}
\end{theorem}
\begin{proof} By Theorem~\ref{theorem:13}  we have
\begin{align*}
\sum\limits_{n=0}^\infty 	S^{\gamma,\alpha,\beta}_{n,k,\ell}\frac{x^n}{n!}&=	\sum_{n=0}^{\infty}\left[
		\sum_{i=0}^n\sum_{j=0}^k\binom{n}{i}\sts{i}{j}_{\gamma}^{> \ell}\sts{n-i}{k-j}_{\alpha,\beta}^{\leq \ell}
	\right]\frac{x^n}{n!} \\&=\sum_{j=0}^k	\sum_{m=0}^\infty\sts{m}{k}_{\gamma}^{>\ell}\frac{x^m}{m!}\sum_{t=0}^\infty\sts{t}{k-j}_{\alpha,\beta}^{\leq \ell}\frac{x^t}{t!}\\&=
		\frac{{e^{\gamma x}}}{k!}\sum_{j=0}^k{\binom{k}{j}}
			\left(e^x-\sum\limits_{i=0}^l\frac{x^i}{i!}\right)^j\left(
			\sum\limits_{i=1}^l\frac{(\beta|\alpha)_{i}}{\beta}
			\frac{x^i}{i!}\right)^{k-j}
        \\
        &= \frac{e^{\gamma x}}{k!}\left(e^x+\sum_{i=1}^\ell\frac{(\beta|\alpha)_i}{\beta}\frac{x^i}{i!}-\sum_{i=0}^{\ell}\frac{x^i}{i!}\right)^k
	\end{align*}
We changed the order of summation, substituted the generating functions from previous theorems, and finally applied the binomial theorem. 

\end{proof}

By Theorem~\ref{theorem:20} we have the following.
\begin{theorem}For $n,k,\ell, \alpha, \beta,\gamma$ are non-negative integers such that  $n\geq k\ell$,  we have
	\begin{align*}
S^{\gamma,\alpha,\beta}_{n,k,\ell}=\sum_{\underset{r_{k+1}\geq0}{\underset{r_1,r_2,\ldots,r_k\geq l}{ r_1+r_2+\cdots+r_{k+1}=n}}}\binom{n}{r_1,r_2,\ldots,r_{k+1}}S^{\gamma,\alpha,\beta}_{r_{k+1},0,\ell}\prod_{i=1}^{k}	S^{0,\alpha,\beta}_{i,1,\ell}.
	\end{align*}
\end{theorem}

By differentiating \eqref{equation:1} we have the following combinatorial identity.
\begin{theorem}For $r,\alpha\geq 0$, $n\geq k\ell$ and $\beta,k,\ll\in \mathbb{N}$ we have
	\begin{align*}
		S^{\gamma,\alpha,\beta}_{n+1,k,\ell}=\gamma S^{\gamma,\alpha,\beta}_{n,k,\ell}+\sum_{i=0}^n\binom{n}{i}S^{\gamma,\alpha,\beta}_{i,k,\ell}S^{0,\alpha,\beta}_{n-i+1,1,\ell}.
	\end{align*}
\end{theorem}
\section{Asymptotics}
Here we derive some asymptotic results for the numbers $S^{\gamma,\alpha,\beta}_{n,k,\ell}$. The method we use is that developed in \cite{hsu1990power,hsu1991kind, nkonkobe2020combinatorial}. Let $\sigma(n)$ denote the set of partitions of integer $n$ represented by $1^{k_1}2^{k_2}\cdots n^{k_n}$ such that 
$1k_1+2k_2+\cdots+ nk_n=n$, $k_i\geq 0$, $i=1,\ldots, n$ and $k=k_1+k_2+\cdots k_n$.
The parameter $k$ denotes the number of parts of the partition. Let $\sigma(n,k)$ denote the set of partitions of $n$ with $k$ parts.  Let  $\phi(t)$ denote the formal power series $\sum_{n=0}^\infty a_nt^n$ over $\mathbb{C}$ with $\phi(0)=a_0=1$.

For every $j$ $(0\leq j\leq n)$ let $B(n,j)$ be equal to the following sum
\begin{align*}
B(n,j) = \sum_{1^{k_1}2^{k_2}\cdots n^{k_n}\in \sigma(n,n-j)}\frac{a_1^{k_1}a_2^{k_2}\cdots a_n^{k_n}}{k_1!k_2!\cdots k_n!}.
\end{align*}
Let $[t^n](\phi(t))^{\lambda}$ denote the coefficient of $t^n$ in $(\phi(t))^{\lambda}$. Then 
\begin{align*}
\frac{1}{(\lambda)_n}[t^n](\phi(t))^{\lambda}  = \sum_{j=0}^{m} \frac{B(n,j)}{(\lambda-n+j)_{j}}+ o\left(\frac{B(n,s)}{(\lambda-n+s)_{s}}\right).
\end{align*}

\begin{theorem}\label{theorem:200} Given $n\geq0$,	
	\begin{align*}
		\frac{S^{{\gamma k},\alpha,\beta}_{n,k,\ell}}{(k)_nn!}=\sum\limits_{\ell=0}^m\frac{B(n,\ell)}{(k-n+\ell)_{\ell}}+o\left(
			\frac{B(n,m)}{(k-n+m)_m}
		\right),
	\end{align*}
	where \begin{align*}B(n,j)=\sum\limits_{\sigma(n,n-j)}\prod\limits_{i=1}^n\frac{(k!)^{x_i}}{x_i!}\begin{bmatrix}
			\frac{S^{\gamma,\alpha,\beta}_{i,k,\ell}}{i!}
		\end{bmatrix}^{x_i},\end{align*} and $n=o(\sqrt{|k|})$ as  $|k|\rightarrow \infty$.	
\end{theorem}

For the numbers $S^{{\gamma},\alpha,\beta}_{n,k,\ell}$ the following identity holds (see \cite{hsu1990power}):

We now compute a few values of the numbers $B(n,j)$. We have

\begin{align*}
	B(n,0)=&\frac{1}{n!}\begin{Bmatrix}\frac{k!S^{\gamma,\alpha,\beta}_{1,k,\ell}}{1!}\end{Bmatrix}^{n},\\
	B(n,1)=&\frac{1}{(n-2)!}\begin{Bmatrix}\frac{k!S^{\gamma,\alpha,\beta}_{1,k,\ell}}{1!}\end{Bmatrix}^{n-2}\begin{Bmatrix}\frac{k!S^{\gamma,\alpha,\beta}_{2,k,\ell}}{2!}\end{Bmatrix},\\
	B(n,2)=&\frac{1}{(n-3)!}\begin{Bmatrix}\frac{k!S^{\gamma,\alpha,\beta}_{1,k,\ell}}{1!}\end{Bmatrix}^{n-3}\begin{Bmatrix}\frac{k!S^{\gamma,\alpha,\beta}_{3,k,\ell}}{3!}\end{Bmatrix}\\&+\frac{1}{2!(n-4)!}\begin{Bmatrix}\frac{k!S^{\gamma,\alpha,\beta}_{1,k,\ell}}{1!}\end{Bmatrix}^{n-4}\begin{Bmatrix}\frac{k!S^{\gamma,\alpha,\beta}_{2,k,\ell}}{2!}\end{Bmatrix}^2,\\
	B(n,3)=&\frac{1}{(n-4)!}\begin{Bmatrix}
		\frac{k!S^{\gamma,\alpha,\beta}_{1,k,\ell}}{1!}
	\end{Bmatrix}^{n-4}\begin{Bmatrix}
	\frac{k!S^{\gamma,\alpha,\beta}_{4,k,\ell}}{4!}
	\end{Bmatrix}\\&+\frac{1}{(n-5)!}\begin{Bmatrix}
	\frac{k!S^{\gamma,\alpha,\beta}_{1,k,\ell}}{1!}
	\end{Bmatrix}^{n-5}\begin{Bmatrix}
		\frac{k!S^{\gamma,\alpha,\beta}_{2,k,\ell}}{2!}
	\end{Bmatrix}\begin{Bmatrix}
	\frac{k!S^{\gamma,\alpha,\beta}_{3,k,\ell}}{3!}
	\end{Bmatrix}\\&+\frac{1}{3!(n-6)!}\begin{Bmatrix}
	\frac{k!S^{\gamma,\alpha,\beta}_{1,k,\ell}}{1!}
	\end{Bmatrix}^{n-6}\begin{Bmatrix}
	\frac{k!S^{\gamma,\alpha,\beta}_{2,k,\ell}}{2!}
	\end{Bmatrix}^3.
\end{align*}

Thus, we have the following result

\begin{align*}	\frac{S^{\gamma k,\alpha,\beta}_{n,k,\ell}}{(k)_nn!}\sim\: & (k)_nB(n,0)+(k)_{n-1}B(n,1)+(k)_{n-2}B(n,2)\\&+(k)_{n-3}B(n,3).\end{align*}

\section*{Acknowledgements}
The second author would like to thank the university of the Witwatersrand for Financial Support through the FR\&IC Start-up Research Grant.

\end{document}